\documentclass[11pt]{amsart}
\usepackage{amssymb,amsthm,amsmath}
\usepackage[numbers,sort&compress]{natbib}
\usepackage{color}
\usepackage{graphicx}
\usepackage{tikz}
\usepackage{enumerate}
\usepackage{float}
\numberwithin{equation}{section}
\title[ ]{H\"older continuity  of the integrated density of states for  Extended Harper's Model with Liouville frequency }

\author[W. Jian]{Wenwen Jian}
\address{School of Mathematical Sciences, Fudan University, Shanghai 200433, P. R. China} \email{wwjian16@fudan.edu.cn}
\author[Y. Shi]{Yunfeng Shi}
\address{School of Mathematical Sciences, Fudan University, Shanghai 200433, P. R. China} \email{yunfengshi13@fudan.edu.cn}


\theoremstyle{plain}
\newtheorem{theorem}{Theorem}[section]

\newtheorem{lemma}[theorem]{Lemma}

\theoremstyle{definition}
\newtheorem{definition}[theorem]{Definition}

\begin{document}


\begin{abstract}
In this paper, we study the non-self-dual extended Harper's model with a Liouville frequency. Based on the work of \cite{SY}, we show that the integrated density of states (IDS for short) of the model  is $\frac{1}{2}$-H$\ddot{\text{o}}$lder continuous. An an application, we also obtain the Carleson homogeneity of the spectrum.
\end{abstract}
\maketitle
\section{Introduction and main results}
Let us consider the  extended Harper's model (EHM for short), which is given  by
\begin{equation}\label{h1}
(H_{\lambda,\alpha,x}u)_n=c(x+n\alpha)u_{n+1}+\overline{c}(x+(n-1)\alpha)u_{n-1}+2\cos{2\pi(x+n\alpha)}u_n,
\end{equation}
 where $u=\{u_n\}_{n\in\mathbb{Z}}\in \ell^2(\mathbb{Z})$ and
 \begin{eqnarray*}
 c(x)=c_{\lambda}(x)=\lambda_1e^{-2\pi i(x+\frac{\alpha}{2})}+\lambda_2+\lambda_{3}e^{2\pi i(x+\frac{\alpha}{2})},\\ \overline{c}(x)=\overline{c}_{\lambda}(x)=\lambda_1e^{2\pi i(x+\frac{\alpha}{2})}+\lambda_2+\lambda_{3}e^{-2\pi i(x+\frac{\alpha}{2})}.
 \end{eqnarray*}
 We call $\lambda=(\lambda_1,\lambda_2,\lambda_3)\in\mathbb{R}^3_+$ the coupling, $\alpha\in\mathbb{R}\setminus\mathbb{Q}$ the frequency and $x\in\mathbb{R}$ the phase. The EHM was originally proposed by Thouless \cite{Thouless} and if $\lambda_1=\lambda_3=0$ it reduces to the famous almost Mathieu operator (AMO for short). Physically, the EHM describes the influence of a transversal magnetic field of flux $\alpha$ on a single tight-binding electron in a 2-dimensional crystal layer (see \cite{Thouless,AJM}).


\begin{figure}[H]
 \centering
 \includegraphics[width=8.0cm]{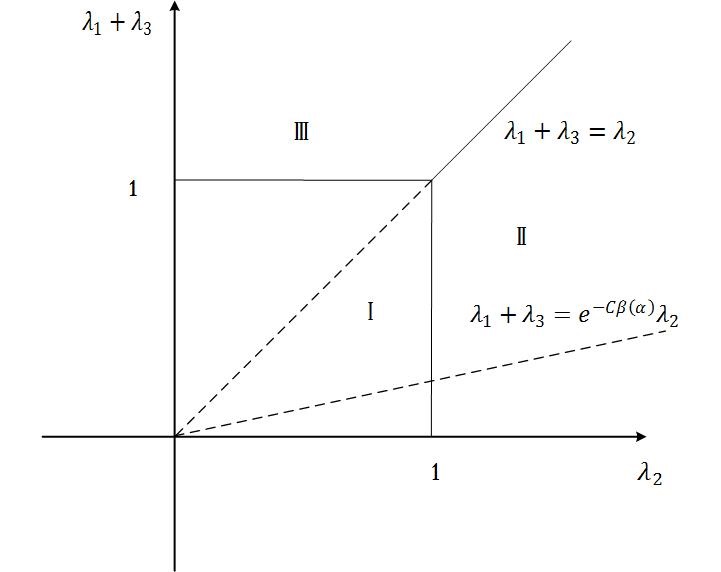}
\caption{}\label{Fig}
\end{figure}
For irrational frequency $\alpha$, the spectrum does not depend on $x$  and we denote it by $ \Sigma_{\lambda ,\alpha}$. Actually, the properties of $\Sigma_{\lambda,\alpha}$ rely heavily on $\lambda,\alpha$. In general, we split the coupling region into three parts (see \textsc{Figure} 1):
  \begin{eqnarray*}
  &&\mathrm{I}=\{(\lambda_1,\lambda_2,\lambda_3)\in\mathbb{R}^3_+:0<\max\{\lambda_1+\lambda_3,\lambda_2\}<1\},\\
 &&\mathrm{II}=\{(\lambda_1,\lambda_2,\lambda_3)\in\mathbb{R}^3_+:0<\max\{\lambda_1+\lambda_3,1\}<\lambda_2\},\\
  &&\mathrm{III}=\{(\lambda_1,\lambda_2,\lambda_3)\in\mathbb{R}^3_+:0<\max\{\lambda_2,1\}<\lambda_1+\lambda_3\}.
  \end{eqnarray*}
According to the duality map $\sigma:(\lambda_1,\lambda_2,\lambda_3)=\lambda\rightarrow\overline{\lambda}= (\frac{\lambda_3}{\lambda_2},\frac{1}{\lambda_2},\frac{\lambda_1}{\lambda_2})$, region $\mathrm{I}$ and region $\mathrm{II}$ are dual to each other and region $\mathrm{III}$ is the self-dual regime. Note that region $\mathrm{I}$ is the regime of positive Lyapunov exponent. When considering $\alpha\in\mathbb{R}\setminus\mathbb{Q}$, we call $\alpha$ a Liouville frequency if $\beta(\alpha)>0$, where
\begin{equation}\label{beta}
\beta(\alpha)=\limsup_{k\to \infty}\frac{-\ln\|k\alpha\|_{\mathbb{R}/\mathbb{Z}}}{|k|}
\end{equation}
and $\|x\|_{\mathbb{R}/\mathbb{Z}}=\min\limits_{k\in\mathbb{Z}}|x-k|$. On the contrary,  $\alpha$ is called a Diophantine frequency for $\beta(\alpha)=0$.

In the present paper, we focus on the regularity of the integrated density of states  $\mathcal{N}_{\lambda,\alpha}(\cdot)$ (see subsection \ref{ids} for details) and homogeneity of the spectrum  in the sense of Carleson for EHM.
The first main result of this paper is the following theorem.
\begin{theorem}\label{holder}
 Suppose $0<\beta(\alpha)<\infty$ and $\lambda\in\mathrm{II}$. Then there is an absolute constant $C>0$ such that if $\mathcal{L}_{\overline{\lambda}}> C\beta(\alpha)$, we have for $E_1,E_2\in \mathbb{R}$,
\begin{equation*}
 |\mathcal{N}_{\lambda,\alpha}(E_1)-\mathcal{N}_{\lambda,\alpha}(E_2)| \leq C_{\star}|E_1-E_2|^{\frac{1}{2}},
\end{equation*}
where $C_\star>0$ is a constant depending on $\alpha,\lambda$ and
 \begin{equation}\label{lya}
\mathcal{L}_{\overline{\lambda}}=\ln{\frac{\lambda_2+\sqrt{\lambda_2^2-4\lambda_1\lambda_3}}
{\max\{\lambda_1+\lambda_3,1\}+\sqrt{\max\{\lambda_1+\lambda_3,1\}^2-4\lambda_1\lambda_3}}}.
\end{equation}
\end{theorem}

Consequently, we also have
\begin{theorem}\label{main2}
Assume the conditions of Theorem \ref{holder} hold.
Then for any $\epsilon>0$, there exists $\sigma_{\star}=\sigma_{\star}(\lambda,\alpha,\epsilon)>0$  such that for all $E\in\Sigma_{\lambda,\alpha}$ and $ \sigma\in(0,\sigma_{\star})$, we have
\begin{equation*}\label{hom}
\mathrm{Leb}\left((E-\sigma,E+\sigma)\cap\Sigma_{\lambda ,\alpha}\right)\geq(1-\epsilon)\sigma,
\end{equation*}
where $\mathrm{Leb}(\cdot)$ is the Lebesgue measure.
\end{theorem}

Let us recall some history about the regularity of the IDS for one-frequency quasi-periodic operators first.
On one hand, we consider in the regime of positive Lyapunov exponent. In \cite{GS2001}£¬ Goldstein-Schlag proved H\"older continuity of the IDS for quasi-periodic Schr\"odinger operator with large analytic potential and Diophantine frequency. Later, Bourgain \cite {B2000} showed the IDS for almost Mathieu operator is $(\frac{1}{2}-\epsilon)$-H\"older continuous (for any small $\epsilon>0$) if the coupling $\lambda_2$ is small and the frequency is Diophantine. Recently, Tao-Voda \cite {TV2017} dealt with quasi-periodic Jacobi operators and obtained especially that the IDS for EHM is $(\frac{1}{2}-\epsilon)$-H\"older continuous if the Lyapunov exponent is positive and the frequency is strong Diophantine. On the other hand, in the subcritical regime Amor \cite{Amor} proved that the IDS for Schr\"odinger operator with Diophantine frequency and small (in perturbative sense) analytic potential is $\frac{1}{2}$-H\"older continuous. After that,  Avila-Jitomirskaya \cite{AJ2010} got $\frac{1}{2}$-H\"older continuity of the IDS for almost Mathieu operator for $\lambda_2\neq\pm1$ and Diophantine frequency and they also obtained $\frac{1}{2}$-H\"older continuity of the IDS for Schr\"odinger operator with small (in non-perturbative sense) analytic potential and Diophantine frequency. Subsequently, Avila-Jitomirskaya \cite{AJCMP} established $\frac{1}{2}$-H\"older continuity of the spectral measures for Schr\"odinger operator with small analytic potential and Diophantine frequency. Note that all above mentioned results are in  Diophantine frequency case and You-Zhang \cite {YZ} extended Goldstein-Schlag's results to weak Liouville frequency case.  In \cite{LYJMP} Liu-Yuan  improved Avila-Jitomirskaya's results to Liouville frequency case.  In a recent work  by Cai-Chavaudret-You-Zhou \cite{CCYZ}, they proved $\frac{1}{2}$-H\"older continuity of the IDS for Schr\"odinger operator with small (perturbative) finitely differentiable potential and Diophantine frequency.

There are also many works on the Carleson homogeneity of the spectrum for quasi-periodic operators. In continuous quasi-periodic Schr\"odinger operator with Diophantine frequency case, Damanik and Goldstein \cite{dam1} set up Carleson homogeneity of the spectrum for small analytic potential. Later, in the regime of positive Lyapunov exponent Goldsein-Damanik-Schlag-Voda \cite{gold2015} proved Carleson homogeneity of the spectrum for quasi-peiodic Schr\"odinger operator with Diophantine frequency. In \cite{gold2016}, Goldstein-Schlag-Voda got the Carleson homogeneity of the spectrum for Diophantine multi-frequency quasi-periodic Schr\"odinger operator. Recently, in the subcritical regime Leguil \cite{Leguil} obtained Carleson homogeneity  of the spectrum for quasi-periodic Schr\"odinger operator with Diophantine frequency.  Actually, we remark that  all these results are attached to Diophantine frequency and Liu-Shi in \cite{LS} extended Leguil's results to Liouville frequency case. In \cite{FL}, Fillman-Lukic established Carleson homogeneity of the spectrum for limit-periodic Schr\"odinger operator.

The present paper is organized as follows. In section 2, we give some
basic concepts and notations. In section 3, we will prove $\frac{1}{2}$-H\"older continuity of the IDS by establishing some quantitative almost reducibility results. The proof of Theorem \ref{main2} is included in Appendix B.

\section{Some basic concepts and notations}

\subsection{Cocycle, transfer matrix and Lyapunov exponent} Let $\alpha\in\mathbb{R}\setminus \mathbb{Q}$ and $C^{\omega}(\mathbb{R}/\mathbb{Z}, \mathcal{B})$ be the set of all analytic mappings from $\mathbb{R}/\mathbb{Z}$ to some Banach space $(\mathcal{B},\|\cdot\|)$. By a cocycle, we mean a pair $(\alpha,A)\in (\mathbb{R}\setminus\mathbb{Q})\times C^{\omega}(\mathbb{R}/\mathbb{Z},{\rm SL}(2,\mathbb{R}))$
and we can regard it as a dynamical system  on $(\mathbb{R}/\mathbb{Z})\times \mathbb{R}^2$ with
\begin{equation*}
(\alpha,A):(x,v)\longmapsto (x+\alpha, A(x)v),\ (x,v)\in (\mathbb{R}/\mathbb{Z})\times \mathbb{R}^2.
\end{equation*}
For $k>0$, we
define  the $k$-step transfer matrix as
\begin{equation*}
A_k(x)=\prod\limits_{l=k}^{1}A(x+(l-1)\alpha),
\end{equation*}
and the Lyapunov exponent for $(\alpha,A)$ as
\begin{equation*}
\mathcal{L}(\alpha,A)=\lim_{k\to +\infty  }\frac{1}{k}\int_{\mathbb{R}/\mathbb{Z}}\ln\|A_k(x)\|\mathrm{d}x
=\inf_{k>0}\frac{1}{k}\int_{\mathbb{R}/\mathbb{Z}}\ln\|A_k(x)\|\mathrm{d}x.
\end{equation*}

\subsection{Spectral measures and the IDS}\label{ids}
Let $H$ be a bounded self-adjoint operator on $\ell^2(\mathbb{Z})$. Then $(H-z)^{-1}$ is analytic in $\mathbb{C}\setminus \Sigma(H)$, where $\Sigma(H)$ is the spectrum of $H$, and we have for $f\in \ell^2(\mathbb{Z})$,
$$\Im \langle (H-z)^{-1}f,f\rangle=\Im z\cdot\|(H-z)^{-1}f\|_{\ell^2(\mathbb{Z})}^2, $$
where $\langle\cdot,\cdot\rangle$ is the usual inner product in $\ell^2(\mathbb{Z})$. Thus
$\phi_f(z)=\langle (H-z)^{-1}f,f\rangle$
is an analytic function in the upper half plane with $\Im\phi_{f}\geq 0$ ($\phi_f$ is the so-called Herglotz function).
Therefore,  one has a representation
\begin{equation*}
\phi_f(z)=\int_{\mathbb{R}}\frac{1}{x-z}\text{d}\mu^{f}(x),
\end{equation*}
where $\mu^f$ is the spectral measure associated to vector $f$. Alternatively, for any Borel set $\Omega\subseteq\mathbb{R}$,
\begin{equation*}
\mu^f(\Omega)=\langle\mathbb{E}(\Omega)f,f\rangle,
\end{equation*}
where $\mathbb{E}$ is the corresponding spectral projection of $H$.

Denote by $\mu^f_{\lambda, \alpha, x}$ the spectral measure of the operator $H_{\lambda, \alpha, x}$ and vector $f$ as above with $||f||_{\ell^2(\mathbb{Z})}=1$. The IDS $\mathcal{N}_{\lambda, \alpha}: \mathbb{R}\rightarrow [0,1]$ is obtained by averaging the spectral measure $\mu^{f}_{\lambda, \alpha, x}$ with respect to $x$, i.e.,
\begin{equation*}
\mathcal{N}_{\lambda,\alpha}(E)=\int_{\mathbb{R}/\mathbb{Z}}\mu^{f}_{\lambda,\alpha,x}(-\infty,E]\mathrm{d}x.
\end{equation*}
It is a continuous, non-decreasing surjective function and the definition is independent of the choice of $f$.

\subsection{Gap labelling and IDS}
 Each  connected component of  $[E_{\text{min}},E_{\text{max}}]\setminus \Sigma_{\lambda ,\alpha}$ is called a spectral gap, where $ E_{\text{min}} =\min\{E: E\in \Sigma_{\lambda,\alpha}\} $ and $ E_{\text{max}} =\max\{E: E\in \Sigma_{\lambda,\alpha}\}$. By the well-known gap labelling theorem \cite{JM1982,Delyon1983The}, for every spectral gap $G$ there exists unique nonzero integer $m$ such that $\mathcal{N}_{\lambda,\alpha}|_{G}=m\alpha \mod{\mathbb{Z}}$ and
\begin{equation}\label{sgap}
[E_m^-,E_m^+]=\{E_{\text{min}}\leq E\leq E_{\text{max}}:\mathcal{N}_{\lambda,\alpha}(E)=m\alpha\mod{\mathbb{Z}}\}.
\end{equation}

\subsection{Extended Harper's cocycle}
  Recalling (\ref{h1}), for $c(x)\neq 0$ the equation $$H_{\lambda,\alpha,x}u=Eu$$ is equivalent to
  \begin{equation*}
  \left(
           \begin{array}{c }
             u_{k+1} \\
             u_k\\
           \end{array}
         \right)=A_{\lambda,E}(x+k\alpha)\left(
           \begin{array}{c }
             u_{k} \\
             u_{k-1}\\
           \end{array}
         \right),
  \end{equation*}
  where $A_{\lambda,E}(x)=\frac{1}{c(x)}\left[
 \begin{array}{cc}
 E-2\cos2\pi x&-\overline{c}(x-\alpha)\\
c(x)&0\end{array}
\right]$.

Since in general, $A_{\lambda,E}(x)\notin \mathrm{SL}(2,\mathbb{R})$, we need make a few modifications and consider the ``renormalized'' $\mathrm{SL}(2,\mathbb{R})$-cocycle
\begin{eqnarray*}\overline{A}_{\lambda,E}(x)&=&\frac{1}{\sqrt{|c|(x)|c|(x-\alpha)}}\left[
 \begin{array}{cc}
 E-2\cos2\pi x&-|c|(x-\alpha)\\
|c|(x)&0\end{array}
\right]\\
&=&Q_{\lambda}(x+\alpha)A_{\lambda,E}(x)Q^{-1}_{\lambda}(x),
\end{eqnarray*}
 where $|c|(x)=\sqrt{c(x)\overline{c}(x)}$
 \begin{footnote}{$\overline{c}(x)$ is the complex conjugate of $c(x)$ for $x\in\mathbb{R}/\mathbb{Z}$ and its analytic extension for $x\notin\mathbb{R}$.}
 \end{footnote}
 and $Q_\lambda,Q^{-1}_\lambda$ are analytic in $\{x\in\mathbb{C}/\mathbb{Z}:|\Im x|\leq \frac{\mathcal{L}_{\overline{\lambda}}}{4\pi}\}$ if $\mathcal{L}_{\overline{\lambda}}\geq 5\beta(\alpha)$ (see Lemma \ref{lema} of the Appendix for details). We call $(\alpha, \overline{A}_{\lambda,E})$ the extended Harper's cocycle and denote by $\mathcal{L}_{\lambda}(E)=\mathcal{L}(\alpha,\overline{A}_{\lambda,E})$ its Lyapunov exponent. Actually, there is a direct definition of the Lyapunov exponent $\mathcal{L}(\alpha,A_{\lambda,E})$ for $(\alpha,A_{\lambda,E})$ (see \cite{JiM1} for details) and $\mathcal{L}_{\lambda}(E)=\mathcal{L}(\alpha,A_{\lambda,E})$ (ignoring the dependence on $\alpha$).


The Thouless formula relates the Lyapunov exponent to the integrated density of states,
\begin{equation}\label{thouless}
\mathcal{L}_{\lambda}(E)=-\int_{\mathbb{R}/\mathbb{Z}}\ln|c_{\lambda}(x)|\mathrm{d}x
+\int_{\mathbb{R}}\ln|E'-E|\mathrm{d}\mathcal{N}_{\lambda,\alpha}(E').
\end{equation}



\subsection{Aubry duality}
The map $\sigma:\lambda=(\lambda_1,\lambda_2,\lambda_3) \to\overline{\lambda}=(\frac{\lambda_3}{\lambda_2},\frac{1}{\lambda_2},\frac{\lambda_1}{\lambda_2})$ induces the duality between region $\mathrm{I}$ and region $\mathrm{II}$, and we call $H_{\overline{\lambda},\alpha,x}$ the Aubry duality of $H_{\lambda,\alpha,x}$. We have $\Sigma_{\lambda,\alpha}=\lambda_2\Sigma_{\overline{{\lambda}},\alpha}$ for $\alpha\in\mathbb{R}\setminus\mathbb{Q}$.

 Aubry duality expresses an algebraic relation between the families of
operators $ \{H _{\overline{\lambda},\alpha,x}\}_{x\in\mathbb{R}} $  and $ \{ {H} _{\lambda ,\alpha,x}\}_{x\in\mathbb{R}} $
by  Bloch waves, i.e., if
 $ u:\mathbb{R}/\mathbb{Z}\rightarrow\mathbb{C}$ is an $L^2$  function whose Fourier coefficients  $\hat u$ satisfy
  $ H _{\overline{\lambda},\alpha,\theta}\hat{u}=\frac{E}{\lambda_2}\hat{u}$,
then there exist $\theta\in \mathbb{R}$, such that
  $U(x)= \left(
           \begin{array}{c }
             e^{2\pi i \theta }u(x) \\
             u(x-\alpha)\\
           \end{array}
         \right)$
satisfies
\begin{equation}\label{dualre}
A_{\lambda,E}(x)\cdot U(x)=e^{2\pi i \theta}U(x+\alpha).
\end{equation}




\subsection{Some notations}
We briefly comment on  the constants and norms in the following proofs.  Let $C(\alpha)$ be  a large constant depending  on $\alpha$   and $C_{\star}$ (resp. $c_{\star}$)  be a large (resp. small) constant   depending on $\lambda$ and $\alpha$. Define the strip $\Delta_s=\{z\in\mathbb{C}/\mathbb{Z}: |\Im{z}|<s\}$ and let $\|v\|_{s}=\sup\limits_{z\in\Delta_s}\|v(z)\|$, where $v$ is a mapping from $\Delta_s$ to some Banach space $(\mathcal{B},\|\cdot\|)$. In this paper, $ \mathcal{B}$ may be
$\mathbb{C},\ \mathbb{C}^2$ or $ \text{SL}(2,\mathbb{C})$.

\section{$\frac{1}{2}$-H\"older continuity of the IDS}

In this section we will prove the $\frac{1}{2}$-H\"older continuity of the IDS for EHM. To this end, one needs to establish quantitative (almost) reducibility results for the extended Harper's cocycle. Let us begin with some useful definitions and lemmata.
\begin{definition}
Fix $\theta\in\mathbb{R},\epsilon_0>0$. We call $n\in\mathbb{Z}$ an $\epsilon_0$-resonance of $\theta$ if  $$\min\limits_{|k|\leq|n|}{\|2\theta-k\alpha\|_{\mathbb{R}/\mathbb{Z}}}=\|2\theta-n\alpha\|_{\mathbb{R}/\mathbb{Z}}\leq e^{-\epsilon_0 |n|}.$$
\end{definition}
Given $\theta\in\mathbb{R}$, we order all the $\epsilon_0$-resonances of $\theta$ as $0<|n_1|\leq|n_2|<\cdots$. We say $\theta$ is $\epsilon_0$-resonant if the set of all $\epsilon_0$-resonances of $\theta$ is infinite and $\epsilon_0$-non-resonant for otherwise. Supposing $\{0,n_1,\cdots,n_j\}$ is the set of all $\epsilon_0$-resonances of $\theta$, we let $n_{j+1}=\infty$.

\begin{lemma}[Theorem 3.3 of \cite{AJ2010}]\label{arl1}
Let $E\in\Sigma_{\lambda,\alpha}$. Then there exist $\theta=\theta(E)\in\mathbb{R}$ and solution $u$ of $H_{\overline{\lambda},\alpha,\theta}u=\frac{E}{\lambda_2}u$ with $u_0=1,|u_k|\leq1$.
\end{lemma}
Throughout this section we fix $E,\theta=\theta(E)$ and $u$ which are given by Lemma \ref{arl1}.

In the following, we let $C_{2}, C_{1}\ (C_{2}\gg C_{1})$ be  large absolute constants which are bigger than any positive absolute constant $C$. Moreover, we assume $\lambda\in \mathrm{II}$  and
$$ h=\frac{\mathcal{L}_{\overline{\lambda}}}{200\pi},\ \mathcal{L}_{\overline{\lambda}}>C_2\beta(\alpha).$$


From Theorem 3.3 in \cite{SY}, we have
\begin{equation}\label{are2}
|u_k|\leq C_{\star} e^{-2\pi h|k|}, \text{ for }3|n_j| <|k|<\frac{|n_{j+1}|}{3},
\end{equation}
where $\{n_j\}$ is the set of all $C^{2}_1\beta(\alpha)$-resonances of $\theta=\theta(E)$.


\begin{lemma}[Lemma 6.6 of \cite{SY}]\label{te}
 We have
\begin{equation}\label{liu3}
\sup_{0\leq k\leq  e^{\frac{hn}{20}}}\|\overline{A}_k\|_{\frac{h}{20}}\leq C_{\star}e^{C\beta(\alpha)n},
\end{equation}
where $\overline{A}_k(x)$ denotes the $k$-step transfer matrix of $(\alpha, \overline{A}_{\lambda,E})$ and $C>0$ is some absolute constant.
\end{lemma}

\begin{lemma}[Theorem 2.6 of \cite{AA2008}]\label{arll}
Given $\eta>0$, we let $U:\mathbb{C}/\mathbb{Z}\rightarrow \mathbb{C}^{2}$ be analytic in $\Delta_{\eta}$ and satisfy $\delta_1\leq\|U(x)\|\leq \delta_2^{-1} \ for\ \forall x\in\Delta_{\eta}$. Then there exists $B(x):\mathbb{C}/\mathbb{Z}\rightarrow \mathrm{SL}(2,\mathbb{C})$ being analytic in $\Delta_{\eta}$ with first column $U(x)$ and $\|B\|_{\eta}\leq C\delta_1^{-2}\delta_2^{-1}(1-\ln(\delta_1\delta_2))$, where $C>0$ is some absolute constant.
\end{lemma}
For simplicity, we write $n=|n_j|<\infty$ and $N=|n_{j+1}|$ in the following.

Define $I_2=\left[-\lfloor\frac{N}{9}\rfloor,  \lfloor\frac{N}{9}\rfloor \right]$
and $$ U^{I_2}(x)=\left(\begin{array}{c}
                   e^{2\pi i\theta}\sum\limits_{k\in I_2}{u}_k e^{2\pi ikx} \\
                   \sum\limits_{k\in I_2} u_k e^{2\pi ik(x-\alpha)}
                 \end{array}\right),
$$
where $\lfloor x\rfloor$ denotes the integer part of $x\in\mathbb{R}$.
Suppose $U_{\star}^{I_2}(x)=Q_{\lambda}(x)\cdot U^{I_2}(x)$.  Recalling (\ref{dualre}) and (\ref{are2}), we have
\begin{equation}\label{au}
\overline{A}_{\lambda,E}(x)U_\star^{I_2}(x)=e^{2\pi i\theta}U_\star^{I_2}(x+\alpha)+G_{\star}(x),
\end{equation}
where
\begin{equation}\label{g}
\|G_{\star}\|_{\frac{h}{3}}\leq C_{\star}e^{-\frac{h}{10}N}.
\end{equation}
We have the following useful estimate.

\begin{lemma}[Lemma 6.6 in \cite{SY}]\label{ustar}
We have for $n>n(\lambda,\alpha)$,
\begin{equation}\label{are12}
\inf_{x\in \Delta_{\frac{h}{3}}}\|U_\star^{I_2}(x)\|\geq e^{-C\beta(\alpha) n},
\end{equation}
where $C>0$ is some absolute constant.
\end{lemma}

We now turn to the upper bound. From (\ref{are2}) and the definition of $u$ in Lemma \ref{arl1}, one has
\begin{equation}\label{upp1}
\begin{aligned}
\|U_\star^{I_2}(x)\|_{C_{1}\beta(\alpha)} \leq&\  C_{\star}\sum_{|k|\leq 3n}|{u}_{k}|e^{2\pi C_{1}\beta(\alpha)|k|} +C_{\star}\sum_{ 3n < |k|\leq \frac{N}{9}}|{u}_{k}|e^{2\pi C_{1}\beta(\alpha)|k|}\\
\leq&\  C_{\star}e^{CC_{1}\beta(\alpha)n}.
\end{aligned}
\end{equation}

The purpose of the following is to construct quantitative almost reducibility (in  $\mathrm{SL}(2,\mathbb{C})$) results. Suppose now $B(x)$ is as in Lemma \ref{arll}  with $U(x)=U^{I_2}_\star(x)$ and $\eta=C_1\beta(\alpha)$. Then from (\ref{ustar}) (\ref{upp1}) and Lemma \ref{arll}, we obtain
\begin{equation}\label{b}
\|B\|_{C_{1}\beta(\alpha)},\|B^{-1}\|_{C_{1}\beta(\alpha)}\leq C_{\star}e^{CC_{1}\beta(\alpha)n}.
\end{equation}
More precisely, by letting  $B(x)=(U^{I_2}_\star(x),V(x))$ and recalling (\ref{au}), we have
\begin{equation*}
\begin{aligned}
\overline{A}_{\lambda,E}(x)B(x)&=\left[e^{2\pi i\theta}U^{I_2}_\star(x+\alpha)+G_{\star}(x),\overline{A}_{\lambda,E}(x)V(x)\right]\\
        &=B(x+\alpha)\left[
                       \begin{array}{cc}
                         e^{2\pi i\theta} & 0 \\
                         0 & e^{-2\pi i\theta} \\
                       \end{array}
                     \right]
          +\left[G_{\star}(x),\overline{A}_{\lambda,E}(x)V(x)-e^{-2\pi i\theta}V(x+\alpha)\right].
\end{aligned}
\end{equation*}
In other words,
\begin{equation}\label{bab}
\begin{aligned}
B^{-1}(x+\alpha)\overline{A}_{\lambda,E}(x)B(x)=\left[
                       \begin{array}{cc}
                         e^{2\pi i\theta} & 0 \\
                         0 & e^{-2\pi i\theta} \\
                       \end{array}
                      \right]
                 +\left[
                       \begin{array}{cc}
                         \beta_{1}(x) & b(x) \\
                         \beta_{2}(x) & \beta_{3}(x) \\
                       \end{array}
                      \right].
\end{aligned}
\end{equation}
From (\ref{g}) and (\ref{b}), we get
\begin{equation}\label{bt1}
\|\beta_{1}\|_{C_{1}\beta(\alpha)}, \|\beta_{2}\|_{C_{1}\beta(\alpha)} \leq C_{\star}e^{-\frac{h}{20}N},
\end{equation}
and
\begin{equation}\label{bc}
\|b\|_{C_{1}\beta(\alpha)}\leq C_{\star}e^{CC_{1}\beta(\alpha)n}.
\end{equation}
By taking determinant on (\ref{bab}) and noting $\overline{A}_{\lambda,E},B \in \mathrm{SL}(2,\mathbb{C})$, one has
\begin{equation}\label{bt3}
\|\beta_{3}\|_{C_{1}\beta(\alpha)}\leq \|b\|_{C_{1}\beta(\alpha)}\|\beta_{2}\|_{C_{1}\beta(\alpha)} + \|\beta_{1}\|_{C_{1}\beta(\alpha)} \leq C_{\star}e^{-\frac{h}{30}N}.
\end{equation}

Actually, one can obtain the following refinement.
\begin{theorem}\label{amr}
Under the previous assumptions, there exists $\Phi(x):\mathbb{C}/\mathbb{Z}\rightarrow \mathrm{SL}(2,\mathbb{C})$ being analytic in $\Delta_{\frac{1}{2}{C_{1}\beta(\alpha)}}$ with
 $\|\Phi\|_{\frac{1}{2}{C_{1}\beta(\alpha)}}\leq C_{\star}e^{CC_{1}\beta(\alpha)n}$
 such that
\begin{equation}\label{are10}
\Phi^{-1}(x+\alpha)\overline{A}_{\lambda,E}(x)\Phi(x)=\left[\begin{array}{cc}e^{2\pi i\theta}&0\\
0&e^{-2\pi i\theta}\end{array}\right]+\left[\begin{array}{cc}\beta'_1{(x)}&b'(x)\\
\beta'_2{(x)}&\beta'_3{(x)}\end{array}\right]
\end{equation}
with
\begin{equation}\label{beta'}
\|\beta'_1\|_{\frac{1}{2}C_{1}\beta(\alpha)},\|\beta'_2\|_{\frac{1}{2}C_{1}\beta(\alpha)},
\|\beta'_3\|_{\frac{1}{2}C_{1}\beta(\alpha)}\leq C_{\star}e^{-\frac{h}{50}N},
\end{equation}
and
\begin{equation}
 \|b'\|_{\frac{1}{2}C_{1}\beta(\alpha)}\leq C_{\star}e^{-\frac{1}{20}C^2_{1}\beta(\alpha)n}.
\end{equation}
\end{theorem}

\begin{proof}
We assume $n>n(\lambda,\alpha)$,  otherwise this theorem is trivial. Recalling (\ref{bab}), we can
write $b(x)=b^{r}(x)+b^{l}(x)+b^{h}(x)$, where
 $b^{l}(x)=\sum_{|k|\leq C_{1}n,\text{ } k\neq n_j}\hat{b}_{k}e^{2\pi i kx}$, $b^{r}(x)=\hat{b}_{n_j}e^{2\pi i n_jx}$
and $b^{h}(x)=\sum_{|k|> C_{1}n}\hat{b}_{k}e^{2\pi i kx}$.
Then by (\ref{bc}),
\begin{equation}\label{bh}
\|b^{h}\|_{\frac{1}{2}C_{1}\beta(\alpha)}
                          \leq \sum_{|k|> C_{1}n} \|b\|_{C_{1}\beta(\alpha)}e^{-\pi C_{1}\beta(\alpha)|k|}\leq C_{\star}e^{-2C^2_{1}\beta(\alpha)n}.
\end{equation}

We then eliminate  the term $b^l(x)$ by solving some homological equation. From the definition of $\beta(\alpha)$ in (\ref{beta}),
we have the small divisor estimate
\begin{equation}\label{k alpha}
\|k\alpha\|_{\mathbb{R}/\mathbb{Z}}\geq C(\alpha)e^{-\frac{3}{2}\beta(\alpha)|k|},\ \ \text{for } k\neq 0.
\end{equation}
Together with the definition of $\epsilon_0$-resonance, one has for $|k|\leq C_{1}n$ and $k\neq n_j$,
\begin{equation}\label{lower}
\|2\theta-k\alpha\|_{\mathbb{R}/\mathbb{Z}} \geq \|(n_j-k)\alpha\|_{\mathbb{R}/\mathbb{Z}}-\|2\theta-n_j\alpha\|_{\mathbb{R}/\mathbb{Z}} \geq c_{\star}e^{-3C_{1}\beta(\alpha)n}.
\end{equation}
Let
$\hat{w}_{k}=-\hat{b}_{k}\frac{e^{-2\pi i\theta}}{1-e^{-2\pi i(2\theta-k\alpha)}}$
for $|k|\leq C_{1}n$ and $k\neq n_j$, and $\hat{w}_{k}=0$ for $|k|> C_{1}n$ or $k= n_j$.
Consequently, the function $w(x)=\sum\limits_{k\in\mathbb{Z}}\hat{w}_ke^{2\pi ik x}$ will satisfy $\|w\|_{\frac{1}{2}C_1\beta(\alpha)}\leq C_\star e^{CC_1\beta(\alpha)n}$ from (\ref{bc}) and (\ref{lower}). If we define
$$ W(x)=  \left[
               \begin{array}{cc}
                      1 & w(x) \\
                      0 & 1 \\
                      \end{array}
                        \right],
$$
then we must have
$$ W^{-1}(x+\alpha)\left[
                       \begin{array}{cc}
                         e^{2\pi i\theta} & b^{l}(x) \\
                         0 & e^{-2\pi i\theta} \\
                       \end{array}
                      \right] W(x)
                 =\left[
                       \begin{array}{cc}
                         e^{2\pi i\theta} &0 \\
                         0 & e^{-2\pi i\theta} \\
                       \end{array}
                      \right],
$$
and
\begin{equation}\label{we}
\|W\|_{\frac{1}{2}C_1\beta(\alpha)}\leq C_\star e^{CC_1\beta(\alpha)n}.
\end{equation}
We now set $\Phi(x)=B(x)W(x)$ and then $\|\Phi\|_{\frac{1}{2}C_1\beta(\alpha)}\leq C_\star e^{CC_1\beta(\alpha)n}$. By direct computation, we have
\begin{equation*}\label{b-1ab}
\begin{aligned}
\Phi^{-1}(x+\alpha)\overline{A}_{\lambda,E}(x)\Phi(x)= Z(x)+\Psi(x),
\end{aligned}
\end{equation*}
with
$$ Z(x)= \left[\begin{array}{cc}
                         e^{2\pi i\theta} & b^{r}(x) \\
                         0 & e^{-2\pi i\theta} \\
                       \end{array}
                      \right]
$$
and
$$\Psi(x)=\left[
             \begin{array}{cc}
             \beta'_{1}(x) & b^{h}(x) \\
             \beta'_{2}(x)  & \beta'_{3}(x)  \\
              \end{array}
          \right]
    =W^{-1}(x+\alpha)\left[
                       \begin{array}{cc}
                         \beta_{1}(x) & b^{h}(x) \\
                         \beta_{2}(x)  & \beta_{3}(x)  \\
                       \end{array}
                      \right] W(x).
$$
Hence we can obtain (\ref{beta'}) and
\begin{equation}\label{psi}
\|\Psi\|_{\frac{1}{2}C_1\beta(\alpha)}\leq C_\star e^{-C_{1}^2\beta(\alpha)n}
\end{equation}
from (\ref{bt1}), (\ref{bt3}), (\ref{bh}) and (\ref{we}).

Thus what remains is to estimate the term $b^r(x)$. For $s\in\mathbb{N}$, we set
\begin{equation*}
Z_{s}(x)=\prod_{k=s-1}^{0}Z(x+k\alpha)=\left[
                                         \begin{array}{cc}
                                           e^{2\pi is\theta} & b^{r}_{s}(x) \\
                                           0 & e^{-2\pi is\theta} \\
                                         \end{array}
                                       \right],
\end{equation*}
where
\begin{equation*}
b^{r}_{s}(x)=\hat{b}_{n_{j}}e^{2\pi i((s-1)\theta+n_{j}x)}\sum_{k=0}^{s-1}e^{-2\pi ik(2\theta-n_{j}\alpha)}.
\end{equation*}
Therefore,
\begin{equation*}
\|b^{r}_{s}\|_0
=\left|\hat{b}_{n_{j}}\frac{\sin \pi s(2\theta-n_{j}\alpha)}{\sin \pi (2\theta-n_{j}\alpha)}\right|
\end{equation*}
if $\sin \pi (2\theta-n_{j}\alpha)\neq 0$, and $\|b^{r}_{s}\|_0=s|\hat{b}_{n_{j}}|$ otherwise.
Noting that
$$2\|x\|_{\mathbb{R}/\mathbb{Z}}\leq\sin \pi\|x\|_{\mathbb{R}/\mathbb{Z}}\leq \pi\|x\|_{\mathbb{R}/\mathbb{Z}},$$
we have for
$0\leq s\leq \frac{1}{2}\|2\theta-n_{j}\alpha\|^{-1}_{\mathbb{R}/\mathbb{Z}}$,
$$\frac{2s}{\pi}|\hat{b}_{n_{j}}|\leq\|b^{r}_{s}\|_0\leq s|\hat{b}_{n_{j}}|.$$
Therefore, for $ 0\leq s\leq \frac{1}{2}\|2\theta-n_{j}\alpha\|^{-1}_{\mathbb{R}/\mathbb{Z}}$,
\begin{equation}\label{zs}
\frac{2s}{\pi}|\hat{b}_{n_{j}}| \leq \|Z_{s}\|_{0}\leq 1+s|\hat{b}_{n_{j}}|\leq C_{\star}(1+s)e^{CC_{1}\beta(\alpha)n}.
\end{equation}
Because of
\begin{equation*}
\begin{aligned}
&\Phi^{-1}(x+s\alpha)\overline{A}_{s}(x)\Phi(x) \\
=\ & Z_{s}(x)+\sum_{k=1}^{s}\ \sum_{s-1\geq j_{1}>j_{2}>\cdots>j_{k}\geq0} \Psi(x+j_{1}\alpha)
\cdots \Psi(x+j_{k}\alpha)\\
                 &\ \ \ \times Z_{s-1-j_{1}}(x+(j_{1}+1)\alpha)Z_{j_{1}-j_{2}-1}(x+(j_{2}+1)\alpha)\cdots Z_{j_{k}}(x)
\end{aligned}
\end{equation*}
and combing with (\ref{psi}) and (\ref{zs}), we have for $s\sim e^{\frac{1}{10}C_1^2\beta(\alpha)n}<\frac{1}{2}\|2\theta-n_j\alpha\|_{\mathbb{R}/\mathbb{Z}}^{-1}$,
\begin{equation*}
\begin{aligned}
\|\overline{A}_{s}\|_{0} \geq&\  \|\Phi\|_{0}^{-2}\left( \|Z_{s}\|_{0}-\sum_{k=1}^{s}\binom{s}{k}
                                                                     \|\Psi\|_{0}^{k} (\max_{0\leq j\leq s-1}\|Z_{j}\|_{0})^{1+k}
                                                  \right)\\
 \geq&\  \|\Phi\|_{0}^{-2}\left( \|Z_{s}\|_{0}-C_\star e^{\frac{1}{10}C_1^2\beta(\alpha)n}\sum_{k=1}^{s}
                                                                      \binom{s}{k}
                                                                     2^{k}e^{-\frac{1}{2}C_1^2\beta(\alpha)nk}                                      \right)\\
      \geq&\   \|\Phi\|_{0}^{-2}\left( \|Z_{s}\|_{0}-C_\star e^{\frac{1}{10}C_1^2\beta(\alpha)n}((1+2e^{-\frac{1}{2}C_1^2\beta(\alpha)n})^s-1)\right)\\                                                        \geq&\  c_{\star}e^{-CC_{1}\beta(\alpha)n}(\|Z_{s}\|_{0}-C_{\star}e^{-\frac{3}{10}C_{1}^{2}\beta(\alpha)n}).
\end{aligned}
\end{equation*}
Thus from Lemma \ref{te} and (\ref{zs}), we have for $s\sim e^{\frac{1}{10}C_1^2\beta(\alpha)n}$,
\begin{equation*}
|\hat{b}_{n_{j}}|\leq C_{\star}e^{-\frac{1}{15}C_{1}^{2}\beta(\alpha)n},
\end{equation*}
and hence
$$\|b'\|_{\frac{1}{2}C_{1}\beta(\alpha)}\leq C_{\star}e^{-\frac{1}{20}C^2_{1}\beta(\alpha)n}.$$
The proof is finished.
\end{proof}

Now we give the proof of Theorem \ref{holder}.

\begin{proof}[{\bf Proof of Theorem \ref{holder}}] If the energy $E$ is in the resolvent set, then $\mathcal{L}_{\lambda}$ is clearly Lipschitz continuous. Thus it is  suffice to consider the case $E\in\Sigma_{\lambda,\alpha}$.
Given $\epsilon>0$, we define $D=\left[
         \begin{array}{cc}
           d^{-1} & 0 \\
           0 & d \\
         \end{array}
       \right]
$ where $d=\|\Phi\|_{\frac{1}{2}C_{1}\beta(\alpha)}\epsilon^{\frac{1}{4}}$ and $\Phi$ is given by Theorem \ref{amr}.
Let $\Phi'(x)=\Phi(x)D$. If $\epsilon \leq c_{\star}e^{-C_{1}^2\beta(\alpha)n}$, we have
\begin{equation}\label{phi'}
\|\Phi'\|_{\frac{1}{2}C_{1}\beta(\alpha)}\leq C_{\star}\epsilon^{-\frac{1}{4}}.
\end{equation}
Set $B'(x)=\Phi'^{-1}(x+\alpha)\overline{A}_{\lambda,E}(x)\Phi'(x)$, then
\begin{equation*}
\begin{aligned}
B'(x)= &\left[
       \begin{array}{cc}
         e^{2\pi i\theta} & 0 \\
         0 & e^{-2\pi i\theta} \\
       \end{array}
     \right]
    + \left[
        \begin{array}{cc}
          \beta'_{1}(x) & d^{2}b'(x) \\
           d^{-2}\beta'_{2}(x) & \beta'_{3}(x) \\
        \end{array}
       \right]\\
\end{aligned}
\end{equation*}
with
\begin{equation*}
 \left\|\beta'_{1}\right\|_{\frac{1}{2}C_{1}\beta(\alpha)}, \left\|\beta'_{3}\right\|_{\frac{1}{2}C_{1}\beta(\alpha)}\leq C_{\star}e^{-\frac{1}{50}hN},
\end{equation*}
\begin{equation*}
\left\|d^2b'\right\|_{\frac{1}{2}C_{1}\beta(\alpha)}\leq C_{\star}e^{-\frac{1}{50}C_{1}^{2}\beta(\alpha)n}\epsilon^{\frac{1}{2}},
\end{equation*}
and
$$\left\|d^{-2}\beta'_{2}\right\|_{\frac{1}{2}C_{1}\beta(\alpha)}\leq C_{\star}e^{-\frac{1}{100}hN}\epsilon^{-\frac{1}{2}}.$$
If $\epsilon\geq C_{\star}e^{-\frac{1}{100}hN}$, then
 $$\left\|d^{-2}\beta'_{2}\right\|_{\frac{1}{2}C_{1}\beta(\alpha)} \leq C_{\star}\epsilon^{\frac{1}{2}},$$
and
\begin{equation}\label{b'}
\left\|B'\right\|_{\frac{1}{2}C_{1}\beta(\alpha)}\leq 1+ C_{\star}\epsilon^{\frac{1}{2}}.
\end{equation}
As a result, for $ C_{\star}e^{-\frac{1}{100}hN}\leq \epsilon \leq c_{\star}e^{-C_{1}^2\beta(\alpha)n}$,
\begin{equation*}
\mathcal{L}_\lambda(E)=\mathcal{L}(\alpha,B')\leq \ln\left\|B'\right\|_{\frac{1}{2}C_{1}\beta(\alpha)}
     \leq\ln\left( 1+ C_{\star}\epsilon^{\frac{1}{2}}\right)\leq C_{\star}\epsilon^{\frac{1}{2}}.
\end{equation*}

Define $$I_{j}:=\{\epsilon\in\mathbb{R}: C_{\star}e^{-\frac{1}{100}h|n_{j+1}|}\leq \epsilon \leq  c_{\star}e^{-C_{1}^2\beta(\alpha)|n_{j}|}\}.$$
Then for any small $\epsilon_{0}>0$, there exists $j_{0}\in \mathbb{Z^{+}}$ such that $[0,\epsilon_{0}]\subset \bigcup_{j\geq j_{0}}I_{j}$.
Let $\epsilon=|E-E'|\in [0,\epsilon_{0}]$ with $E'\in \mathbb{C}$.
Then by (\ref{phi'}) and (\ref{b'}), one has
\begin{equation*}
\begin{aligned}
  \mathcal{L}_\lambda(E')&= \ \mathcal{L}\left(\alpha,\Phi'^{-1}(x+\alpha)\overline{A}_{\lambda,E'}(x)\Phi'(x)\right)\\
   &\leq\  \ln\left\|B'+\Phi'^{-1}(x+\alpha)\left(\overline{A}_{\lambda,E'}(x)-\overline{A}_{\lambda,E}(x)\right)\Phi'(x)\right\|_{\frac{1}{2}C_1\beta(\alpha)}\\
  &\leq \ \ln\left(1+C_{\star}\epsilon^{\frac{1}{2}}\right) \leq C_{\star}\epsilon^{\frac{1}{2}}.
\end{aligned}
\end{equation*}
Hence,
\begin{equation} \label{lyae}
  |\mathcal{L}_\lambda(E') - \mathcal{L}_\lambda({E})|\leq C_\star|E'-E|^{\frac{1}{2}}.
\end{equation}

From the Thouless formula (\ref{thouless}),
we have
\begin{equation*}
\begin{aligned}
\left|\mathcal{L}\left(\alpha,\overline{A}_{\lambda,E+i\epsilon}\right)-\mathcal{L}_\lambda(E)\right|
   = &\ \frac{1}{2}\int\ln\left(1+\frac{\epsilon^{2}}{(E-E')^{2}}\right)\mathrm{d}\mathcal{N}_{\lambda,\alpha}(E')\\
   \geq &\ \frac{1}{2}\ln 2
   \left(\mathcal{N}_{\lambda,\alpha}(E+\epsilon)-\mathcal{N}_{\lambda,\alpha}(E-\epsilon)\right).
\end{aligned}
\end{equation*}
Thus recalling (\ref{lyae}), we obtain
$$\mathcal{N}_{\lambda,\alpha}(E+\epsilon)-\mathcal{N}_{\lambda,\alpha}(E-\epsilon) \leq C_{\star}\epsilon^{\frac{1}{2}},$$
which means precisely that $\mathcal{N}_{\lambda,\alpha}$ is $\frac{1}{2}$-H\"{o}lder continuous. This finishes the proof of Theorem \ref{holder}.
\end{proof}


\appendix
\section{}
\begin{lemma}\label{lema}
Let $0<\beta(\alpha)<\infty$ and $\lambda\in\mathrm{II}$. If $\mathcal{L}_{\overline{\lambda}}\geq 5\beta(\alpha)$, then there are analytic mapping $Q_{\lambda}$ from $\Delta_{\frac{\mathcal{L}_{\overline{\lambda}}}{4\pi}}$ to $M_2(\mathbb{C})$ and  its inverse $Q^{-1}_{\lambda}$ which is analytic in the same region, such that for all $x\in \Delta_{\frac{\mathcal{L}_{\overline{\lambda}}}{4\pi}}$,
$$Q_\lambda^{-1}(x+\alpha)A_{\lambda,E}(x)Q_{\lambda}(x)=\overline{A}_{\lambda,E}(x),$$
where $M_2(\mathbb{C})$ denotes the space of all $2\times 2$ complex matrices.
\end{lemma}

\begin{proof}
Let
\begin{equation*}\begin{aligned}
\epsilon_{\star}=
\ \min\left\{{\frac{\lambda_2 +\sqrt{\lambda_2^2 -4\lambda_1 \lambda_3}}{2\lambda_1} , \ \frac{\lambda_2 +\sqrt{\lambda_2^2 -4\lambda_1 \lambda_3}}{2\lambda_3} }\right\}.
\end{aligned}\end{equation*}
Then as $\lambda_2>\lambda_1+\lambda_3$, we have for any $\epsilon\in\mathbb{R}$ with $|\epsilon|<\epsilon_\star$,

\begin{eqnarray}
\label {app1}&&\lambda_2 -(\lambda_1 e^{2\pi \epsilon}+\lambda_3 e^{-2\pi \epsilon})>0,\\
\label{app2}&&\lambda_2 -(\lambda_1 e^{-2\pi \epsilon}+\lambda_3 e^{2\pi \epsilon})>0.
\end{eqnarray}
Thus $\Re c(x+i\epsilon)>0, \Re \overline{c}(x+i\epsilon)>0$ for any $x\in\mathbb{R}/\mathbb{Z}$. We have showed that $c(x),\  \overline{c}(x)$  have no zeros on $\Delta_{\frac{\mathcal{L}_{\overline{\lambda}}}{2\pi}}$. Recalling (\ref{app1}) and (\ref{app2}) again, the rotation numbers of $c(x),\ \overline{c}(x)$ on $\Delta_{\frac{\mathcal{L}_{\overline{\lambda}}}{2\pi}}$ are identically vanishing. Consequently, there are single-valued analytic functions $g_1(x)=\log|c(x)|+i\arg c(x)$ and $g_2(x)=\log|\overline{c}(x)|+i\arg \overline{c}(x)$ on $\Delta_{\frac{\mathcal{L}_{\overline{\lambda}}}{2\pi}}$ such that $c(x)=e^{g_1(x)}, \ \overline{c}(x)=e^{g_2(x)}$.

Noting for $x\in\mathbb{R}/\mathbb{Z}$,
\begin{equation*}
\Re {c(1-\alpha-x)}
=\Re {c(x)}, \Im {c(1-\alpha-x)}=-\Im {c(x)},
\end{equation*}
then we have
\begin{equation*}\label{arg}
\begin{aligned}
\int_{\mathbb{R}/\mathbb{Z}}\arg {c(x)}\text{d}x =&\int_{-\frac{\alpha}{2}}^{\frac{1}{2}-\frac{\alpha}{2}}\text{arg }c(x)\text{d}x + \int_{\frac{1}{2}-\frac{\alpha}{2}}^{1-\frac{\alpha}{2}}\text{arg }c(x) \mathrm{d}x\\
=&-\int_{-\frac{\alpha}{2}}^{\frac{1}{2}-\frac{\alpha}{2}}\text{arg }c(1-\alpha-x)\text{d}x + \int_{\frac{1}{2} -\frac{\alpha}{2}}^{1-\frac{\alpha}{2}}\text{arg }c(x)\text{d}x\\
=&\ 0.
\end{aligned}
\end{equation*}
Similarly, $\int_{\mathbb{R}/\mathbb{Z}}\text{arg }\overline{c}(x)\text{d}x=0$. Hence $\widehat{g_1-g_2}(0)=\int_{\mathbb{R}/\mathbb{Z}}\left(g_1(x)-g_2(x)\right)\mathrm{d}x=0$  and  the function $f(x)=\sum\limits_{k\in\mathbb{Z}}\hat{f}_ke^{2\pi kix}$  will solve the equation
$$2f(x+\alpha)-2f(x)=g_1(x)-g_2(x),$$
where $\hat{f}_0=0$ and $\hat{f}_k=\frac{\widehat{g_1-g_2}(k)}{2(e^{2\pi ki\alpha}-1)}, k\neq0$.
Because of the small divisor estimate (\ref{k alpha}) and $\mathcal{L}_{\overline{\lambda}}\geq 5\beta(\alpha)$, $f(x)$ must be analytic on $\Delta_{\frac{\mathcal{L}_{\overline{\lambda}}}{4\pi}}$. Thus $c(x)=|c|(x)e^{f(x+\alpha)-f(x)}, \ \overline{c}(x)=|c|(x)e^{-f(x+\alpha)+f(x)}$ for all $x\in \Delta_{\frac{\mathcal{L}_{\overline{\lambda}}}{4\pi}}$.

 Let \[Q_\lambda(x)=e^{f(x)}\sqrt{|c|(x-\alpha)}
\left[
       \begin{array}{cc}
         1 & 0 \\
         0 & \sqrt{\frac{\overline{c}(x-\alpha)}{c(x-\alpha)}}\\
       \end{array}
     \right].\]
Then  the proof follows (the detail computations are similar to that of \cite{Han}).

\end{proof}

\section{Carleson Homogeneity: proof of Theorem \ref{main2}.}
In this appendix, we will complete the proof of Theorem \ref{main2} and this  follows from the H\"older continuity of the IDS together with the exponential decay of the lengths of the spectral gaps. For the convenience of readers, we include the details in the following.
\begin{lemma}[Theorem 1.1 of \cite{SY}]\label{gap}
Let $\alpha\in\mathbb{R}\setminus\mathbb{Q}$ with $0\leq\beta(\alpha)<\infty$ and $E_m^-,E_m^+$ be given by $\mathrm{(\ref{sgap})}$.  Then there exists absolute constant $C>1$ such that, if  $\lambda\in \mathrm{II}$ and $\mathcal{L}_{\overline{\lambda}}>C\beta(\alpha)$, one has for $|m|\geq m_{\star}$,
\begin{equation*}\label{upp}
E_m^+-E_m^-\leq  e^{-C^{-1}\mathcal{L}_{\overline{\lambda}}|m|},
\end{equation*}
where $m_\star$ is a positive constant only depending  on $\lambda,\alpha$  and $\mathcal{L}_{\overline{\lambda}}$ is given by $\mathrm{(\ref{lya})}$.
\end{lemma}

\begin{lemma}\label{hl2}
Let     $G_m=(E_m^-,E_m^+)$ for  $m\in\mathbb{Z}\setminus\{0\}$ and $G_0=(-\infty,E_{\min})$. Then for $m'\neq m\in\mathbb{Z}\setminus\{0\}$ with $|m'|\geq|m|$, we have
\begin{equation}\label{hle2}
\mathrm{dist}(G_m,G_{m'})=\inf\limits_{x\in G_m,x'\in G_{m'}}{|x-x'|}\geq c_{\star} e^{-6\beta(\alpha)|m'|},
\end{equation}
and for $m\in\mathbb{Z}\setminus\{0\}$
\begin{equation}\label{hle3}
\mathrm{dist}(G_m,G_0)\geq c_{\star} e^{-6\beta(\alpha)|m|}.
\end{equation}
\end{lemma}

\begin{proof}
From the small divisor condition (\ref{k alpha}), one has
\begin{equation}\label{kalpha}
\begin{aligned}
\|(m-m')\alpha\|_{\mathbb{R}/\mathbb{Z}}\geq& C(\alpha)e^{-3\beta(\alpha)|m'|}
\end{aligned}
\end{equation}
for $|m'|\geq |m|$.

Without loss of generality, we assume $E_m^+\leq E_{m'}^-$. By Theorem \ref{holder}, (\ref{sgap}) and (\ref{kalpha}), we have
\begin{equation*}
\begin{aligned}
\mathrm{dist}(G_m,G_{m'})= &\  |E_{m'}^--E_m^+|\\
\geq& \left(\frac{1}{C_{\star}}\left|\mathcal{N}_{\lambda ,\alpha}(E_{m'}^-)-\mathcal{N}_{\lambda ,\alpha}(E_m^+)\right|\right)^2\\
\geq& \ c_{\star} \|(m-m')\alpha\|_{\mathbb{R}/\mathbb{Z}}^2,\\
\geq&\  c_{\star}e^{-6\beta(\alpha)|m'|}
\end{aligned}
\end{equation*}
which completes the proof of  (\ref{hle2}).  The proof of (\ref{hle3}) is similar.
\end{proof}

Now we can give the proof of Theorem \ref{main2}.

\begin{proof}[{\bf Proof of Theorem \ref{main2}}]
Assume $0<\sigma\leq \sigma_{\star}(\lambda,\alpha,\epsilon)$.
For $E\in\Sigma_{\lambda,\alpha}$ and $ \sigma $, let
$$\mathcal{R}(E,\sigma)=\{m\in\mathbb{Z}\setminus\{0\}: (E-\sigma,E+\sigma)\cap G_m\neq \emptyset\}.$$
Define $m_0\in\mathbb{Z}\setminus\{0\}$ with $|m_0|=\min\limits_{m\in\mathcal{R}(E,\sigma)}|m|$. For any $m\in\mathcal{R}(E,\sigma)$, one has
\begin{equation*}
\mathrm{dist}(G_m,G_{m_0})\leq 2\sigma.
\end{equation*}

We first assume $(E-\sigma,E+\sigma)\cap G_0=\emptyset$. Recalling (\ref{hle2}), we have for any $m\in\mathcal{R}(E,\sigma)$ with $m\neq m_0$,
\begin{equation*}
2\sigma\geq c_{\star}e^{-6\beta(\alpha)|m|},
\end{equation*}
that is 
\begin{equation}\label{hle7}
|m|\geq \frac{-\ln{(C_{\star}\sigma)}}{6\beta(\alpha)}.
\end{equation}
Then by Lemma \ref{gap}, we obtain
\begin{equation}\label{leb}
\begin{aligned}
&\sum\limits_{m\in\mathcal{R}(E,\sigma),m\neq m_0}\mathrm{Leb}((E-\sigma,E+\sigma)\cap G_m)\\
\leq\  &\sum\limits_{m\in\mathcal{R}(E,\sigma),m\neq m_0}(E_m^+-E_m^-)\\
\leq \ & \sum\limits_{|m|\geq\frac{-\ln{(C_{\star}\sigma)}}{6\beta(\alpha)}}C_{\star}e^{-C^{-1}\mathcal{L}_{\overline{\lambda}}|m|}\\
\leq \ &\epsilon\sigma.
\end{aligned}
\end{equation}
On the other hand, $E\in\Sigma_{\lambda ,\alpha}$ implies $E\notin G_{m_0}$. Thus  we have
\begin{equation}\label{hle9}
\mathrm{Leb}((E-\sigma,E+\sigma)\cap G_{m_0})\leq\sigma.
\end{equation}
In this case,   (\ref{leb}) and (\ref{hle9}) implies
\begin{equation*}\label{hle10}
\begin{aligned}
&\ \mathrm{Leb}((E-\sigma,E+\sigma)\cap \Sigma_{\lambda,\alpha})\\
\geq&\ 2\sigma-\mathrm{Leb}((E-\sigma,E+\sigma)\cap G_{m_0})\\
&\ \ -\sum\limits_{m\in\mathcal{R}(E,\sigma),m\neq m_0}\mathrm{Leb}((E-\sigma,E+\sigma)\cap G_m)\\
\geq&\ 2\sigma-\sigma-\epsilon\sigma\geq(1-\epsilon)\sigma.
\end{aligned}
\end{equation*}

In the case $(E-\sigma,E+\sigma)\cap G_0\neq\emptyset$, we have  $$0< E_m^--E_{\min}\leq2\sigma$$
for any $m\in\mathcal{R}(E,\sigma)$. Thus, (\ref{hle7}) also holds for any $  m\in \mathcal{R}(E,\sigma)$ by (\ref{hle3}).
From the proof of (\ref{leb}), we have
\begin{equation}\label{hle11}
\sum\limits_{m\in\mathcal{R}(E,\sigma)}\mathrm{Leb}((E-\sigma,E+\sigma)\cap G_m)\leq \epsilon\sigma.
\end{equation}
Noticing that  $E\in\Sigma_{\lambda,\alpha}$ and  $E\notin G_0$, one has
\begin{equation}\label{liu7}
   \mathrm{Leb}((E-\sigma,E+\sigma)\cap G_0)\leq\sigma.
\end{equation}
By
  (\ref{hle11})  and \eqref{liu7}, we obtained
\begin{equation*}
\begin{aligned}
&\ \mathrm{Leb}((E-\sigma,E+\sigma)\cap \Sigma_{\lambda,\alpha})\\
\geq&\ 2\sigma-\mathrm{Leb}((E-\sigma,E+\sigma)\cap G_0)\\
&\ \  \ -\sum\limits_{m\in\mathcal{R}(E,\sigma)}\mathrm{Leb}((E-\sigma,E+\sigma)\cap G_m)\\
\geq & \ 2\sigma-\sigma-\epsilon\sigma\geq(1-\epsilon)\sigma.
\end{aligned}
\end{equation*}
Putting all the cases together, we complete the proof of    Theorem \ref{main2}.
\end{proof}

\footnotesize
\bibliographystyle{abbrv} 

\end{document}